\documentclass{amsart}
\usepackage{amssymb}
\usepackage{graphicx}
\usepackage[hidelinks]{hyperref}

\newtheorem{theorem}{Theorem}[section]
\newtheorem{lemma}[theorem]{Lemma}
\newtheorem{proposition}[theorem]{Proposition}
\numberwithin{equation}{section}

\newcommand{\R}{\mathbb R}
\newcommand{\Pp}{\mathcal P}
\newcommand{\Q}{\mathcal Q}

\begin{document}

\title[Minimality of wide strips]
{Sufficiently wide strips uniquely minimize the planar periodic fractional perimeter}

\author[J. Seo]{Juneyoung Seo}
\address{Innovation Center for Atomic Science, DAU G-LAMP Project Group,
Dong-A University, Busan 49315, Republic of Korea}
\email{juneys.dau@gmail.com}
\thanks{This research was supported by Global - Learning \& Academic research institution for Master’s$\cdot$PhD students, and Postdocs (LAMP) Program of the National Research Foundation of Korea (NRF) grant funded by the Ministry of Education (RS-2025-25440216).}


\date{}

\begin{abstract}
We study the fractional perimeter per period of planar sets with prescribed horizontal period and area per period. We prove that sufficiently wide strips are the unique minimizers among all measurable competitors, up to vertical translation and null sets. For period one, a half-width of at least $14$ suffices for every fractional exponent $s\in(0,1)$. The proof first replaces each vertical section by a centered interval of the same length. We then establish a quantitative lower bound for the perimeter excess in terms of the deviation of the resulting half-width profile from its mean. The main difficulty is that the perimeter of a strip is a concave function of its width. We show that, at large mean width, the interaction cost of unequal sections dominates the corresponding concavity deficit and leaves a positive remainder for every nonconstant profile. Estimates at small and large horizontal separations make the sufficient width bound independent of $s$. The argument applies to arbitrary nonnegative integrable profiles, including unbounded profiles and profiles that vanish on sets of positive measure.
\end{abstract}

\maketitle

\section{Introduction and main result}
The fractional perimeter measures the interaction between a set and its complement through a kernel that decays as a power of the distance. For sets that repeat in one direction, the corresponding isoperimetric problem asks which shapes minimize this interaction per period at a prescribed area per period. We prove that, in the plane, sufficiently large area forces every minimizer to be a straight strip. The sufficient area threshold, expressed relative to the period, is independent of the fractional exponent.

Fix $s\in(0,1)$ and let $I=(0,1)$. Throughout the paper, periodicity means period one in the first coordinate. For a measurable periodic set $E\subset\R^2$, define its fractional perimeter per period by
\begin{equation}\label{e:perimeter}
 \Pp_s(E)=
 \int_{E\cap(I\times\R)}
 \int_{\R^2\setminus E}
 |X-Y|^{-2-s}\,dY\,dX.
\end{equation}
The outer integral is restricted to one period cell, while each point interacts with the whole complement, including all other cells. This is the periodic formulation studied in \cite{CCM, DDPV}, with the fractional-perimeter kernel of \cite{CRS}. We prescribe
\[
 |E\cap(I\times\R)|=2a
\]
and compare $E$ with the horizontal strip
\[
 S_a=\R\times(-a,a),
\]
which has the same area per period.

D\'avila, del Pino, Dipierro, and Valdinoci \cite{DDPV} constructed minimizers in a restricted class of cylindrically symmetric periodic sets whose radius profiles are even and nonincreasing on half a period. They also asked whether straight cylinders minimize at large volumes. Cabr\'e, Csat\'o, and Mas \cite{CCM} subsequently established existence and symmetry among all measurable periodic competitors and formulated the conjecture that straight cylinders are the only minimizers for sufficiently large volumes. In dimension two, these cylinders are strips. Our main theorem establishes this conjecture in the plane.

\begin{theorem}\label{t:main}
For every $s\in(0,1)$ and every $a\ge14$, each measurable periodic set $E\subset\R^2$ with $|E\cap(I\times\R)|=2a$ satisfies
\[
 \Pp_s(E)\ge\Pp_s(S_a).
\]
Equality holds if and only if $E=S_a+(0,z_0)$ up to a null set for some $z_0\in\R$.
\end{theorem}

For a prescribed horizontal period $L>0$ and area $m$ per period, the sufficient condition becomes $m\ge28L^2$, and the minimizing strip has half-width $m/(2L)$. Indeed, simultaneous dilation of the set and period by $L$ multiplies area by $L^2$ and perimeter by $L^{2-s}$. The bound $14$ is not optimized, and the theorem makes no claim about the precise area at which strips become minimizing.

The behavior at small volumes is different. The restricted minimizers of \cite[Theorem 3]{DDPV} are close in measure to periodic arrays of balls as the volume tends to zero. Moreover, straight cylinders fail to minimize at sufficiently small volumes, with a threshold uniform for $s$ bounded away from zero \cite[Theorem 1.1(iv)]{CCM}.

Related results concern critical points and their stability. For smooth sets, the Euler--Lagrange equation under the area constraint requires the boundary to have constant nonlocal mean curvature. Nonstraight periodic planar solutions bifurcating from strips were constructed in \cite{CFSW}. Stability means that the second variation is nonnegative under periodic variations preserving area; a stability radius for straight cylinders was announced in \cite[Theorem 1.3]{CCM}. In the even, transversely symmetric planar class, Bruera \cite[Corollary 1.7]{Bruera2026} proves that stable constrained critical points sufficiently close to a fixed strip in $C^{1,\alpha}$, for suitable $\alpha>s$, are themselves strips. This holds for $s$ sufficiently close to one, and for every $s\in(0,1)$ when the reference half-width differs from the first bifurcation radius. Such local rigidity does not determine whether a strip has less perimeter than every competitor of the same area. Theorem~\ref{t:main} provides that global comparison for large area, without identifying its sufficient width bound with a stability or bifurcation radius. Existence, symmetry, and planar bifurcation results for anisotropic kernels are developed in \cite{AlcoverBruera}.

Our proof reduces the global comparison to an inequality for functions of one variable. Given $E$, replace each vertical section
\[
 E_x=\{z\in\R:(x,z)\in E\}
\]
by the centered interval of the same length. This operation, called transverse Steiner symmetrization, produces
\[
 E^\star=E_u,
 \qquad
 E_u=\{(x,z)\in\R^2:|z|<u(x)\},
 \qquad
 u(x)=\tfrac12|E_x|
\]
for almost every $x\in I$. We identify profiles on $I$ with their periodic extensions. The area constraint gives $u\ge0$ and $\int_Iu\,dx=a$, while the rearrangement inequality yields
\[
 \Pp_s(E^\star)\le\Pp_s(E)
\]
by \cite[Lemma 3.1]{CCM}. It therefore suffices to compare $E_u$ with the constant profile $u\equiv a$.

The difficulty is that symmetrization leaves an arbitrary integrable half-width profile. In particular, averaging the energies of strips of different widths does not favor the constant profile: direct computation gives
\[
 \Pp_s(S_a)=b_sa^{1-s},
\]
where $b_s>0$ depends only on $s$, and concavity implies
\[
 b_s\int_I u^{1-s}\,dx
 \le b_sa^{1-s}.
\]
The profile representation of the perimeter supplies a lower bound consisting of this average strip energy and a nonnegative interaction cost between unequal widths. We show that, when the mean width is sufficiently large, the interaction cost exceeds the concavity deficit
\[
 b_s\left(a^{1-s}-\int_I u^{1-s}\,dx\right)
\]
by a quantity that controls the variation of $u$.

\begin{proposition}\label{p:profile}
For every $s\in(0,1)$ and every nonnegative periodic $u\in L^1(I)$ with mean $a=\int_Iu\,dx\ge14$,
\begin{equation}\label{e:uniform-profile}
 \Pp_s(E_u)-\Pp_s(S_a)
 \ge
 \frac{1}{8(1-s)}
 \int_I
 \frac{|u(x)-a|^2}{1+|u(x)-a|}
 \,dx.
\end{equation}
\end{proposition}

The integrand on the right is quadratic for small deviations and grows linearly for large deviations. Thus the remainder is finite for every integrable profile and vanishes precisely when $u=a$ almost everywhere. For an arbitrary competitor $E$, symmetrization gives the same lower bound with $u(x)=|E_x|/2$, so the estimate controls variation in the lengths of its vertical sections. If equality holds in Theorem~\ref{t:main}, these lengths must be constant and equality must also hold in symmetrization. The latter forces the sections to be intervals with a common center, giving uniqueness up to vertical translation and null sets. Figure~\ref{fig:reduction} illustrates the two comparisons.

\begin{figure}[ht]
\centering
\includegraphics[width=\textwidth]{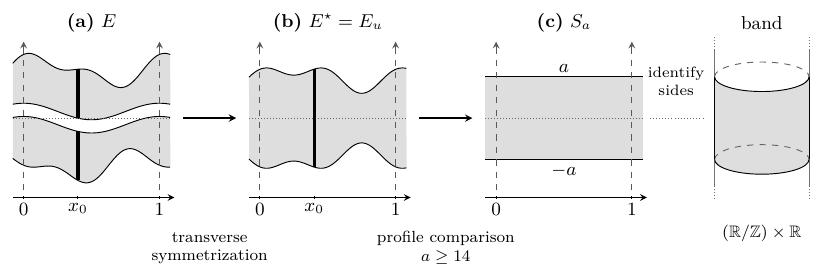}
\caption{The reduction at fixed area $2a$: transverse Steiner symmetrization gives $\Pp_s(E)\ge\Pp_s(E^\star)$, and Proposition~\ref{p:profile} gives $\Pp_s(E^\star)\ge\Pp_s(S_a)$ for $a\ge14$. The marked sections in (a) and (b) have length $2u(x_0)$. Matching arrows identify the sides of each period cell; dotted horizontal lines mark $z=0$. The inset shows the strip as a band on the cylinder.}
\label{fig:reduction}
\end{figure}

To prove Proposition~\ref{p:profile}, we first bound the concavity deficit in terms of pairwise differences $|u(x)-u(y)|$. We then estimate the interaction cost separately at horizontal separations smaller and larger than one period. At small separations, a dyadic argument supplies the factor $(1-s)^{-1}$ needed as $s\uparrow1$. At large separations, a bound on the interaction across translates of the period cell retains the coefficient needed as $s\downarrow0$. These complementary estimates allow a single sufficient width bound for the whole range $s\in(0,1)$. Section~\ref{s:proofs} develops the profile representation and these estimates, and Section~\ref{ss:mainproof} combines them through a pointwise comparison. All estimates apply directly to nonnegative $L^1$ profiles, including unbounded profiles and profiles with empty sections, without regularity or closeness to a strip.

\section{Profile comparisons}\label{s:proofs}
\subsection{Symmetrization and the profile representation}\label{ss:reduction}
The periodized-kernel formulation in \cite{DDPV} agrees with \eqref{e:perimeter}, since periodicity and Tonelli's theorem give
\[
 \Pp_s(E)=\int_{E\cap(I\times\R)}\int_{(I\times\R)\setminus E}
 \sum_{k\in\mathbb Z}|X-Y-(k,0)|^{-2-s}\,dY\,dX.
\]

For $t\ne0$, set
\[
 K_t(z)=(t^2+z^2)^{-(2+s)/2}.
\]
The transverse rearrangement inequality \cite[Lemma 3.1]{CCM}, after dilation from period $2\pi$ to period $1$, gives
\begin{equation}\label{e:rearrangement}
 \Pp_s(E^\star)\le\Pp_s(E).
\end{equation}
For positive finite area and finite perimeter, equality holds precisely when $E=E^\star+(0,z_0)$ up to a null set.

For completeness, we recall why equality forces a single vertical translation. Write $E_x=\{z:(x,z)\in E\}$, and let $E_x^\star$ be its centered rearrangement. For almost every $(x,y)\in I\times\R$ with $x\ne y$, the Riesz inequality gives the finite nonnegative deficit
\[
 \delta(x,y)=
 \int_{E_x^\star}\int_{E_y^\star}K_{x-y}(z-w)\,dw\,dz
 -\int_{E_x}\int_{E_y}K_{x-y}(z-w)\,dw\,dz\ge0.
\]
For each such pair, both the set interaction and the complement interaction are bounded by $|E_x|\|K_{x-y}\|_1<\infty$. 
Write the latter as $C(x,y)$ and its rearranged counterpart as
$C^\star(x,y)$. Since corresponding sections have equal lengths,
we obtain the finite identity
\[
 C(x,y)=|E_x|\|K_{x-y}\|_1
 -\int_{E_x}\int_{E_y}K_{x-y}(z-w)\,dw\,dz
 =C^\star(x,y)+\delta(x,y).
\]
Integrating this nonnegative identity by Tonelli, without subtracting total set interactions, yields
\[
 \Pp_s(E)=\Pp_s(E^\star)+\int_I\int_\R\delta(x,y)\,dy\,dx.
\]
For equal finite perimeters, $\delta=0$ almost everywhere. The strict Riesz inequality, in the form used in \cite[Lemma 3.1]{CCM}, makes any such pair of sections of positive measure intervals with the same center, since $K_{x-y}$ is strictly decreasing in $|z|$. Put $A_+=\{x\in I:0<|E_x|<\infty\}$, which has positive measure. Choose $x_0\in A_+$ from the full-measure set supplied by Fubini, so that almost every $y\in I$ is a good partner. A partner in $A_+$ first makes $E_{x_0}$ an interval with a unique center $z_0$; almost every other positive section then has this same center. Zero sections agree with their translated rearrangements modulo null sets. Tonelli and periodicity extend this sectionwise agreement to all of $\R^2$ modulo a planar null set.

For centered interval sections, the interaction can be computed by integrating the kernel twice. Define
\[
 G_t(r)=\int_0^{|r|}(|r|-z)K_t(z)\,dz
\]
and the nonnegative interaction
\[
 \Q_s(u)=2\int_\R\int_I G_t(u(x)-u(x-t))\,dx\,dt.
\]
The term $\Q_s(u)$ measures the interaction between unequal section widths and vanishes on constant profiles. Set
\begin{equation}\label{e:constants}
 c_s=\int_\R(1+t^2)^{-(2+s)/2}\,dt,\qquad
 b_s=\frac{2^{2-s}c_s}{s(1-s)}.
\end{equation}
Comparison with the exponents $3/2$ and $1$ gives $2\le c_s\le\pi$.
The planar profile representation appears in Alviny\`a's thesis \cite{Alvinya}. We derive the following lower bound directly for all nonnegative $L^1$ profiles, allowing infinite perimeter.

\begin{lemma}\label{l:representation}
Every nonnegative periodic $u\in L^1(I)$ satisfies
\[
 \Pp_s(E_u)\ge\Q_s(u)+b_s\int_Iu^{1-s}\,dx.
\]
\end{lemma}

\begin{proof}
For $t\ne0$, define
\[
 H_t(r)=\int_0^\infty\min\{r,z\}K_t(z)\,dz,
 \qquad r\ge0.
\]
This function is nonnegative and concave. Direct integration gives
\[
 \int_\R K_t(z)\,dz=c_s|t|^{-1-s},\qquad
 G_t(r)+H_t(r)=\frac{c_s}{2}|t|^{-1-s}r\quad(r\ge0).
\]
Since $G_t$ is even and $G_t''=K_t$, integration over two intervals gives, for $p,q\ge0$,
\[
 \int_{-p}^p\int_{-q}^qK_t(z-w)\,dw\,dz
 =2G_t(p+q)-2G_t(p-q).
\]
Consequently,
\[
 \int_{-p}^p\int_{\R\setminus(-q,q)}K_t(z-w)\,dw\,dz
 =2G_t(p-q)+2H_t(p+q)+c_s|t|^{-1-s}(p-q).
\]
For fixed $t\ne0$, both $G_t$ and $H_t$ have at most linear growth. We may therefore set $p=u(x)$ and $q=u(x-t)$ and integrate in $x\in I$. Periodicity cancels the last term because
\[
 \int_I(u(x)-u(x-t))\,dx=0.
\]
This cancellation is made before integrating in $t$, since $|t|^{-1-s}$ is not integrable near zero. Tonelli's theorem then gives
\begin{equation}\label{e:representation}
 \Pp_s(E_u)=\Q_s(u)+2\int_\R\int_I
 H_t(u(x)+u(x-t))\,dx\,dt.
\end{equation}

For $z>0$ and $r\ge0$, direct integration, with $t=z\tau$ in the first identity, gives
\[
 \int_\R K_t(z)\,dt
 =c_sz^{-1-s},\quad
 \int_0^\infty\min\{r,z\}z^{-1-s}\,dz
 =\frac{r^{1-s}}{s(1-s)}.
\]
Thus
\begin{equation}\label{e:H-integral}
 2\int_\R H_t(r)\,dt=b_s(r/2)^{1-s},\qquad r\ge0.
\end{equation}
In particular, \eqref{e:representation} with $u\equiv a$ gives $\Pp_s(S_a)=b_sa^{1-s}$. If $a=\int_Iu$, Jensen's inequality gives $\int_IH_t(u(x)+u(x-t))\,dx\le H_t(2a)$. Thus the second term in \eqref{e:representation} is at most $b_sa^{1-s}<\infty$, and finite perimeter is equivalent to $\Q_s(u)<\infty$. Finally, concavity gives
\[
 H_t(p+q)\ge\tfrac12H_t(2p)+\tfrac12H_t(2q).
\]
After integration in $x$, the two terms on the right agree by periodicity. Using \eqref{e:H-integral}, we obtain the bound needed in \eqref{e:representation}:
\[
 2\int_\R\int_I H_t(u(x)+u(x-t))\,dx\,dt
 \ge 2\int_I\int_\R H_t(2u(x))\,dt\,dx
 =b_s\int_Iu^{1-s}\,dx.\qedhere
\]
\end{proof}

\subsection{A direct bound for the concavity deficit}\label{ss:estimates}
For a nonnegative periodic $u\in L^1(I)$ of mean $a$, define the concavity deficit
\begin{equation}\label{e:deficits}
 D_s(u)=a^{1-s}-\int_Iu^{1-s}\,dx\ge0.
\end{equation}
Lemma~\ref{l:representation} gives
\begin{equation}\label{e:cost-loss}
 \Pp_s(E_u)-\Pp_s(S_a)\ge\Q_s(u)-b_sD_s(u).
\end{equation}
The next lemma bounds this deficit by pairwise profile differences, with only linear growth for large differences.

\begin{lemma}\label{l:direct-deficit}
For $a>0$ and $r\ge0$, define
\begin{equation}\label{e:B}
 B_{s,a}(r)=
 \begin{cases}
 sa^{-1-s}r^2,&0\le r\le a,\\
 (1+s)a^{-s}r-r^{1-s},&r\ge a.
 \end{cases}
\end{equation}
Every nonnegative periodic $u\in L^1(I)$ of mean $a>0$ satisfies
\begin{equation}\label{e:direct-deficit}
 2D_s(u)\le\int_I\int_I B_{s,a}(|u(x)-u(y)|)\,dy\,dx.
\end{equation}
\end{lemma}

\begin{proof}
For $r\ge0$, set
\[
 g(r)=r^{1-s}-(1+s)a^{-s}r+sa^{-1-s}r^2.
\]
Convexity of $r\mapsto r^{-s}$ on $(0,\infty)$ gives $g(r)\ge0$ for $r>0$, and $g(0)=0$. Moreover, $g'(a)=0$ and
\[
 g''(r)=2sa^{-1-s}-s(1-s)r^{-1-s}>0\qquad(r\ge a),
\]
so $g$ is nondecreasing on $[a,\infty)$.

Let $p,q\ge0$ and $r=|p-q|$. If $r\le a$, nonnegativity of $g$ gives
$g(p)+g(q)\ge sa^{-1-s}r^2-B_{s,a}(r)=0$.
If $r\ge a$, the larger of $p,q$ is at least $r$, and hence
\[
 g(p)+g(q)\ge g(r)=sa^{-1-s}r^2-B_{s,a}(r).
\]
Expanding $g(p)+g(q)$ therefore yields, in both cases,
\begin{equation}\label{e:two-point-deficit}
 p^{1-s}+q^{1-s}
 \ge(1+s)a^{-s}(p+q)-2sa^{-1-s}pq-B_{s,a}(|p-q|).
\end{equation}
By substituting $p=u(x),q=u(y)$ and integrating this over $I^2$, we have
\[
 2\int_Iu^{1-s}\,dx
 \ge2a^{1-s}-\int_I\int_I B_{s,a}(|u(x)-u(y)|)\,dy\,dx.\qedhere
\]
\end{proof}

\subsection{A dyadic estimate at small horizontal separations}\label{ss:dyadic}
We now estimate the part of $\Q_s(u)$ arising from horizontal separations shorter than one period. Set
\[
 \Psi(r)=\frac{r^2}{1+|r|},
\]
an even convex function that is nondecreasing in $|r|$ and has at most linear growth.
Telescoping and convexity transfer profile differences to successively smaller scales; compare \cite[Lemma 3.1]{Ponce}. Summing over these scales gives the following estimate.

\begin{lemma}\label{l:dyadic}
Every periodic $u\in L^1(I)$ satisfies
\begin{equation}\label{e:dyadic-pair}
 2\int_{0<|t|<1}\int_I G_t(u(x)-u(x-t))\,dx\,dt
 \ge\frac{1}{1-2^{s-1}}
 \int_I\int_I\Psi(u(x)-u(y))\,dy\,dx.
\end{equation}
\end{lemma}
\begin{proof}
Since $s<1$, direct integration gives
\[
 2G_1(r)
 \ge2\int_0^{|r|}(|r|-z)(1+z^2)^{-3/2}\,dz
 =2\bigl(\sqrt{1+r^2}-1\bigr)
 \ge\Psi(r).
\]
The change of variables $z=|t|w$ gives $G_t(r)=|t|^{-s}G_1(r/|t|)$, and hence
\begin{equation}\label{e:scaled-cost}
 2G_t(r)\ge|t|^{-s}\Psi(r/|t|),\qquad 0<|t|\le1.
\end{equation}

Fix an integer $k\ge0$. For $\rho\in(1/2,1)$, let $n=2^k$ and $t=\rho/n$. Telescoping gives
\[
 \frac{u(x)-u(x-\rho)}{\rho}
 =\frac1n\sum_{j=0}^{n-1}
 \frac{u(x-jt)-u(x-(j+1)t)}{t}.
\]
Convexity, periodicity, and monotonicity of $\Psi$ in the absolute value of its argument imply
\begin{align*}
 \int_I\Psi(u(x)-u(x-\rho))\,dx
 &\le\int_I\Psi\!\left(\frac{u(x)-u(x-\rho)}{\rho}\right)\,dx\\
 &\le\int_I\Psi\!\left(\frac{u(x)-u(x-t)}{t}\right)\,dx.
\end{align*}
Observe that the function $A(\rho)=\int_I\Psi(u(x)-u(x-\rho))\,dx$ is periodic and satisfies $A(-\rho)=A(\rho)$. Hence $A(1-\rho)=A(\rho)$, giving
\[
 2\int_{1/2}^1\int_I\Psi(u(x)-u(x-\rho))\,dx\,d\rho
 =\int_I\int_I\Psi(u(x)-u(y))\,dy\,dx.
\]
On the interval $(2^{-k-1},2^{-k})$, use $t=2^{-k}\rho$, \eqref{e:scaled-cost}, and $\rho^{-s}\ge1$. We obtain
\begin{align*}
 &4\int_{2^{-k-1}}^{2^{-k}}\int_I G_t(u(x)-u(x-t))\,dx\,dt\\
 &\qquad\ge2^{1-k(1-s)}\int_{1/2}^1\int_I
 \Psi(u(x)-u(x-\rho))\,dx\,d\rho\\
 &\qquad=2^{-k(1-s)}\int_I\int_I
 \Psi(u(x)-u(y))\,dy\,dx.
\end{align*}
The positive and negative $t$ contributions agree by translation in $x$ and evenness of $G_t$ in both $t$ and its argument. The nonnegative shell sum is valid even if infinite, and $\sum_{k\ge0}2^{-k(1-s)}=(1-2^{s-1})^{-1}$ proves the claim.
\end{proof}

\subsection{A lower bound from large horizontal separations}\label{ss:long}
To estimate the contribution from larger separations, we sum over translates of the period cell. The resulting kernel bound leads to the function
\begin{equation}\label{e:F}
 F_s(r)=\int_0^r(r-z)\left(z+\frac{3\pi}{4}\right)^{-1-s}\,dz,\qquad r\ge0.
\end{equation}
This function is convex and Lipschitz, with $0\le F_s'(r)\le(3\pi/4)^{-s}/s$.

\begin{lemma}\label{l:long}
Every periodic $u\in L^1(I)$ satisfies
\begin{equation}\label{e:long-bound}
 2\int_{|t|\ge1}\int_I G_t(u(x)-u(x-t))\,dx\,dt
 \ge2c_s\int_I\int_I F_s(|u(x)-u(y)|)\,dy\,dx.
\end{equation}
\end{lemma}

\begin{proof}
We first prove the tail bound
\begin{equation}\label{e:tail-bound}
 2\int_b^\infty K_t(z)\,dt\ge c_s\bigl(z+\tfrac\pi2 b\bigr)^{-1-s}
 \qquad(b>0,\ z\ge0).
\end{equation}
For $z>0$, put $\theta=\arctan(z/b)\in(0,\pi/2)$. The substitution $t=z\cot v$ gives
\[
 \frac{2\int_b^\infty K_t(z)\,dt}{c_sz^{-1-s}}
 =\frac{\int_0^\theta\sin^s v\,dv}{\int_0^{\pi/2}\sin^s v\,dv}.
\]
Concavity of sine gives $\sin(\theta v)\ge(2\theta/\pi)\sin(\pi v/2)$ for $0\le v\le1$. Hence
\[
 \int_0^\theta\sin^s v\,dv
 =\theta\int_0^1\sin^s(\theta v)\,dv
 \ge\left(\frac{2\theta}{\pi}\right)^{1+s}\int_0^{\pi/2}\sin^s v\,dv.
\]
Since $\sin\theta\le\theta$ and $\cos\theta\ge1-2\theta/\pi$,
\[
 \frac{2\theta}{\pi}
 \ge\frac{1}{1+(\pi/2)\cot\theta}
 =\frac{z}{z+(\pi/2)b}.
\]
Combining these inequalities proves \eqref{e:tail-bound} for $z>0$; the case $z=0$ follows by continuity.

For $0<\rho<1$, the numbers $|j+\rho|\ge1$, $j\in\mathbb Z$, form two sequences starting at $1+\rho$ and $2-\rho$. Since $K_t(z)$ decreases with $t>0$, each sum bounds the corresponding tail integral from below. The function $b\mapsto\int_b^\infty K_t(z)\,dt$ is convex. Consequently,
\begin{align}
 \sum_{|j+\rho|\ge1}K_{j+\rho}(z)
 &\ge\int_{1+\rho}^\infty K_t(z)\,dt
      +\int_{2-\rho}^\infty K_t(z)\,dt\notag\\
 &\ge2\int_{3/2}^\infty K_t(z)\,dt
 \ge c_s\left(z+\frac{3\pi}{4}\right)^{-1-s}.
 \label{e:lattice-kernel}
\end{align}
When $\rho=0$ or $1$, both sequences start at $1$, giving a stronger tail bound. For $r\ge0$, multiplying \eqref{e:lattice-kernel} by $r-z$ and integrating over $0<z<r$ gives
\[
 \sum_{|j+\rho|\ge1}G_{j+\rho}(r)\ge c_sF_s(r).
\]
Finally, for fixed $x\in I$, write $t=x-y+j$ with $y\in I$ and $j\in\mathbb Z$. These substitutions cover the disjoint cells $(x+j-1,x+j)$ up to endpoints, with absolute Jacobian one. Periodicity gives $u(x-t)=u(y)$, so Tonelli yields
\[
 2\int_{|t|\ge1}\int_I G_t(u(x)-u(x-t))\,dx\,dt
 =2\int_I\int_I
 \sum_{|x-y+j|\ge1}G_{x-y+j}(u(x)-u(y))\,dy\,dx.
\]
Reindex $j$ to apply the preceding bound to the fractional part of $x-y$. Evenness in the profile difference then proves \eqref{e:long-bound}, with the exterior factor $2$ giving exactly $2c_s$.
\end{proof}

\section{Proof of the main theorem}\label{ss:mainproof}
\begin{proof}[Proof of Proposition~\ref{p:profile}]
Let $u\ge0$ have mean $a\ge14$. Recall that
\[
 D_s(u)=a^{1-s}-\int_Iu^{1-s}\,dx,
 \qquad b_s=\frac{2^{2-s}c_s}{s(1-s)}.
\]
By \eqref{e:cost-loss},
\[
 \Pp_s(E_u)-\Pp_s(S_a)\ge\Q_s(u)-b_sD_s(u).
\]
Thus we must show that the interaction $\Q_s(u)$ controls the concavity deficit $b_sD_s(u)$ and leaves the stated remainder. Since $\Psi(r)=r^2/(1+|r|)\le|r|$ and $\int_I|u-a|\le2a$, that remainder is finite. The case $\Pp_s(E_u)=+\infty$ is therefore immediate; henceforth we assume that the perimeter is finite.
Then $\Q_s(u)<\infty$ by Lemma~\ref{l:representation}.

Write
\[
 J=\int_I\int_I\Psi(u(x)-u(y))\,dy\,dx.
\]
Splitting the integral defining $\Q_s(u)$ into the regions $0<|t|<1$ and $|t|\ge1$ and applying Lemmas~\ref{l:dyadic} and~\ref{l:long}, we obtain
\[
 \Q_s(u)\ge\frac{J}{1-2^{s-1}}
 +2c_s\int_I\int_I F_s(|u(x)-u(y)|)\,dy\,dx.
\]
We will reserve $J/[8(1-s)]$ from the first term and use the remaining terms to control the deficit. The argument below first justifies this division of the coefficient and then proves the required pointwise comparison.

\smallskip
\noindent\emph{Reserving the remainder.}
We claim that
\begin{equation}\label{e:reserve}
 \frac{1}{1-2^{s-1}}\ge\frac1{8(1-s)}+\frac{9c_s}{16(1-s)}.
\end{equation}
To estimate $c_s$, write its integrand as
\[
 (1+t^2)^{-(2+s)/2}
 =\bigl((1+t^2)^{-1}\bigr)^{1-s}
  \bigl((1+t^2)^{-3/2}\bigr)^s.
\]
The two factors before taking these powers have integrals $\pi$ and $2$, respectively. H\"older's inequality, followed by the weighted arithmetic--geometric mean inequality and $\pi<22/7$, therefore gives
\[
 c_s\le\pi^{1-s}2^s\le(1-s)\pi+2s
 <2+\frac87(1-s).
\]
For the coefficient on the left of \eqref{e:reserve}, consider
$h(x)=x(e^x+1)-2(e^x-1)$. Since $h(0)=h'(0)=0$ and $h''(x)=xe^x\ge0$ for $x\ge0$, we have $h(x)\ge0$. Rearranging this inequality yields
\[
 \frac{x}{1-e^{-x}}\ge1+\frac{x}{2}\qquad(x>0).
\]
Taking $x=(1-s)\log2$ and using $\log2<7/10$, we obtain
\[
 \frac{1-s}{1-2^{s-1}}
 \ge\frac1{\log2}+\frac{1-s}{2}
 >\frac{10}{7}+\frac{1-s}{2}.
\]
On the other hand, the bound for $c_s$ gives
\[
 \frac18+\frac{9c_s}{16}<\frac54+\frac9{14}(1-s).
\]
The difference between these last two comparison bounds is
\[
 \left(\frac{10}{7}+\frac{1-s}{2}\right)
 -\left(\frac54+\frac9{14}(1-s)\right)
 =\frac5{28}-\frac{1-s}{7}\ge\frac1{28}>0.
\]
Division by $1-s>0$ proves \eqref{e:reserve}.

\smallskip
\noindent\emph{A pointwise comparison for the deficit.}
Recall the functions appearing in Lemmas~\ref{l:direct-deficit} and~\ref{l:long}:
\[
 B_{s,a}(r)=
 \begin{cases}
 sa^{-1-s}r^2,&0\le r\le a,\\
 (1+s)a^{-s}r-r^{1-s},&r\ge a,
 \end{cases}
\]
\[
 F_s(r)=\int_0^r(r-z)(z+\kappa)^{-1-s}\,dz,
\]
where we set $\kappa=3\pi/4$ for the remainder of the proof. We claim that
\begin{equation}\label{e:combined-cost}
 s(1-s)F_s(r)+\frac{9s}{32}\Psi(r)
 \ge2^{-s}B_{s,a}(r)\qquad(r\ge0,\ a\ge14).
\end{equation}
The normalization is chosen so that multiplication by $2c_s/[s(1-s)]$ gives
\[
 2c_sF_s(r)+\frac{9c_s}{16(1-s)}\Psi(r)
 \ge\frac{b_s}{2}B_{s,a}(r).
\]
Indeed, the coefficient on the right is $2^{1-s}c_s/[s(1-s)]=b_s/2$. After integration, Lemma~\ref{l:direct-deficit}, which states that $2D_s(u)\le\int_I\int_I B_{s,a}(|u(x)-u(y)|)\,dy\,dx$, will give exactly the required control of $b_sD_s(u)$.

To prove \eqref{e:combined-cost}, first note that
\begin{equation}\label{e:Fbounds}
\begin{aligned}
 s(1-s)F_s(r)
 &=(1-s)\kappa^{-s}r-(\kappa+r)^{1-s}+\kappa^{1-s}\\
 &\ge(1-s)\kappa^{-s}r-r^{1-s}.
\end{aligned}
\end{equation}
We compare the two sides of \eqref{e:combined-cost} separately on the two ranges in the definition of $B_{s,a}$. Denote their difference by
\[
 \Delta(r)=s(1-s)F_s(r)+\frac{9s}{32}\Psi(r)-2^{-s}B_{s,a}(r).
\]

For $0<r\le a$, the substitution $z=rt$ gives
\[
 \frac{F_s(r)}{r^2}=\int_0^1(1-t)(\kappa+rt)^{-1-s}\,dt.
\]
The integrand is nonincreasing in $r$, so $F_s(r)\ge r^2F_s(a)/a^2$. Also $\Psi(r)=r^2/(1+r)\ge r^2/(1+a)$. Using the quadratic branch $B_{s,a}(r)=sa^{-1-s}r^2$, we obtain
\[
 \Delta(r)\ge r^2\left(
 \frac{s(1-s)}{a^2}F_s(a)+\frac{9s}{32(1+a)}
 -s2^{-s}a^{-1-s}\right).
\]
Multiplication by $a/r^2$ and \eqref{e:Fbounds} at $r=a$ now yield
\begin{align*}
 \frac{a}{r^2}\Delta(r)
 &\ge\frac{s(1-s)}aF_s(a)+\frac{9s}{32}\frac{a}{1+a}-s2^{-s}a^{-s}\\
 &\ge(1-s)\kappa^{-s}-(1+s2^{-s})a^{-s}
 +\frac{9s}{32}\frac{a}{1+a}.
\end{align*}
Let this last expression be
\begin{equation}\label{e:scalar-gap}
 \Gamma_{s,a}=(1-s)\kappa^{-s}-(1+s2^{-s})a^{-s}
 +\frac{9s}{32}\frac{a}{1+a}.
\end{equation}

For $r\ge a$, monotonicity of $r/(1+r)$ gives $\Psi(r)\ge ar/(1+a)$. Apply \eqref{e:Fbounds} at $r$ and use the second branch of $B_{s,a}$ to obtain
\begin{align*}
 \frac{\Delta(r)}r
 &\ge (1-s)\kappa^{-s}-r^{-s}+\frac{9s}{32}\frac{a}{1+a}
       -2^{-s}\bigl((1+s)a^{-s}-r^{-s}\bigr)\\
 &= (1-s)\kappa^{-s}-2^{-s}(1+s)a^{-s}
       -(1-2^{-s})r^{-s}+\frac{9s}{32}\frac{a}{1+a}\\
 &\ge\Gamma_{s,a}.
\end{align*}
For the last inequality, $r^{-s}\le a^{-s}$ and $1-2^{-s}>0$ imply
$-(1-2^{-s})r^{-s}\ge-(1-2^{-s})a^{-s}$; the coefficients then combine as
$2^{-s}(1+s)+(1-2^{-s})=1+s2^{-s}$.
Since $\Delta(0)=0$, both ranges are settled once $\Gamma_{s,a}\ge0$. By \eqref{e:scalar-gap}, multiplying $\Gamma_{s,a}\ge0$
by $a^s>0$ shows that this condition is equivalent to
\begin{equation}\label{e:parameter-comparison}
 (1-s)\left(\frac{4a}{3\pi}\right)^s+\frac{9s}{32}\frac{a^{1+s}}{1+a}
 \ge1+s2^{-s}.
\end{equation}

\smallskip
\noindent\emph{Verification for $a\ge14$.}
For fixed $s\in(0,1)$, both terms on the left of \eqref{e:parameter-comparison} increase with $a$. Thus it suffices to set $a=14$. The left side then becomes
\[
 f(s)=(1-s)\left(\frac{56}{3\pi}\right)^s+\frac{21}{80}s\,14^s.
\]
We show the stronger bound $f(s)\ge1+s$ by proving that $f$ is convex and that $f(0)=1$, $f'(0)>1$.

For the derivative calculation, put $L=\log(56/(3\pi))$ and $M=\log14$. The elementary bounds
\[
 \frac74<L<\frac95,\qquad M>\frac52,
 \qquad\log\kappa>\frac12
\]
give
\[
 f(0)=1,\qquad
 f'(0)=L-1+\frac{21}{80}>1.
\]
Moreover, $\kappa^s=e^{s\log\kappa}\ge1+s/2$. Differentiating $f$ twice and dividing by the positive factor $(56/(3\pi))^s$, we find
\begin{align*}
 \left(\frac{3\pi}{56}\right)^sf''(s)
 &=(1-s)L^2-2L+\frac{21}{80}\kappa^s(2M+sM^2)\\
 &\ge\frac{49}{16}(1-s)-\frac{18}{5}
 +\frac{21}{80}\left(1+\frac{s}{2}\right)\left(5+\frac{25s}{4}\right)\\
 &=\frac{31}{40}(1-s)+\frac{3s}{320}+\frac{105s^2}{128}>0.
\end{align*}
Hence $f$ lies above its tangent at zero, and
\[
 f(s)\ge f(0)+sf'(0)\ge1+s\ge1+s2^{-s}.
\]
This proves \eqref{e:parameter-comparison}, hence $\Gamma_{s,a}\ge0$ and the pointwise comparison \eqref{e:combined-cost}.

\smallskip
\noindent\emph{Combining the estimates.}
Multiply \eqref{e:combined-cost} by $2c_s/[s(1-s)]$, substitute $r=|u(x)-u(y)|$, and integrate over $I^2$. The coefficient identity noted above and Lemma~\ref{l:direct-deficit} give
\begin{align*}
 &\frac{9c_s}{16(1-s)}J
 +2c_s\int_I\int_I F_s(|u(x)-u(y)|)\,dy\,dx\\
 &\qquad\ge\frac{b_s}{2}\int_I\int_I B_{s,a}(|u(x)-u(y)|)\,dy\,dx
 \ge b_sD_s(u).
\end{align*}
Using \eqref{e:reserve} in the lower bound for $\Q_s(u)$ at the start of the proof, we therefore obtain
\begin{align*}
 \Q_s(u)
 &\ge\frac{J}{8(1-s)}+\frac{9c_s}{16(1-s)}J
 +2c_s\int_I\int_I F_s(|u(x)-u(y)|)\,dy\,dx\\
 &\ge\frac{J}{8(1-s)}+b_sD_s(u).
\end{align*}

It remains to pass from pairwise differences to deviation from the mean. For almost every fixed $x$, convexity of $\Psi$ and $|I|=1$ give
\[
 \int_I\Psi(u(x)-u(y))\,dy
 \ge\Psi\!\left(\int_I(u(x)-u(y))\,dy\right)
 =\Psi(u(x)-a).
\]
Integrating the last inequality in $x$ and returning to \eqref{e:cost-loss}, we conclude that
\[
 \Pp_s(E_u)-\Pp_s(S_a)
 \ge\Q_s(u)-b_sD_s(u)
 \ge\frac{J}{8(1-s)}
 \ge\frac1{8(1-s)}\int_I\Psi(u-a)\,dx,
\]
which is the asserted estimate.
\end{proof}

\begin{proof}[Proof of Theorem~\ref{t:main}]
Let $|E\cap(I\times\R)|=2a$ with $a\ge14$. Since $\Pp_s(S_a)<\infty$, we may assume $\Pp_s(E)<\infty$. By \eqref{e:rearrangement} and Proposition~\ref{p:profile},
\[
 \Pp_s(E)\ge\Pp_s(E^\star)
 \ge\Pp_s(S_a)+\frac{1}{8(1-s)}\int_I\Psi(u-a)\,dx.
\]
This proves minimality. Equality forces $\int_I\Psi(u-a)=0$, hence $u=a$ almost everywhere because $\Psi\ge0$ vanishes only at zero. The same chain forces $\Pp_s(E)=\Pp_s(E^\star)<\infty$. Section~\ref{ss:reduction} then gives $E=S_a+(0,z_0)$ up to a null set. Conversely, each such translate has the same area per period and perimeter as $S_a$.
\end{proof}


\begin{thebibliography}{99}
\bibitem{AlcoverBruera}
F. Alcover and R. Bruera,
\emph{Periodic Delaunay cylinders with constant anisotropic nonlocal mean curvature},
preprint, arXiv:2602.18219 (2026).

\bibitem{Alvinya}
M. Alviny\`a,
\emph{Delaunay cylinders with constant non-local mean curvature},
Master's thesis, Universitat Polit\`ecnica de Catalunya, 2017.

\bibitem{Bruera2026}
R. Bruera,
\emph{On the shape of minimizers for the periodic nonlocal perimeter in $\mathbb R^2$},
preprint, arXiv:2602.18215 (2026).

\bibitem{CCM}
X. Cabr\'e, G. Csat\'o, and A. Mas,
\emph{Existence and symmetry of periodic nonlocal-CMC surfaces via variational methods},
J. Reine Angew. Math. \textbf{804} (2023), 11--40.

\bibitem{CFSW}
X. Cabr\'e, M. M. Fall, J. Sol\`a-Morales, and T. Weth,
\emph{Curves and surfaces with constant nonlocal mean curvature: meeting Alexandrov and Delaunay},
J. Reine Angew. Math. \textbf{745} (2018), 253--280.

\bibitem{CRS}
L. Caffarelli, J.-M. Roquejoffre, and O. Savin,
\emph{Nonlocal minimal surfaces},
Comm. Pure Appl. Math. \textbf{63} (2010), 1111--1144.

\bibitem{DDPV}
J. D\'avila, M. del Pino, S. Dipierro, and E. Valdinoci,
\emph{Nonlocal Delaunay surfaces},
Nonlinear Anal. \textbf{137} (2016), 357--380.

\bibitem{Ponce}
A. C. Ponce,
\emph{An estimate in the spirit of Poincar\'e's inequality},
J. Eur. Math. Soc. \textbf{6} (2004), no. 1, 1--15.

\end{thebibliography}
\end{document}